\documentclass[12pt]{amsart}

\makeatletter
\def\blfootnote{\gdef\@thefnmark{}\@footnotetext}
\makeatother

\usepackage{epsfig}
\usepackage{graphics}
\usepackage{dcpic, pictexwd}

\usepackage[colorlinks=true,
                    linkcolor=blue,
                    urlcolor=blue,
                    citecolor=blue,
                    anchorcolor=blue]{hyperref}
\usepackage{mathtools}
\usepackage{amsmath,amssymb}
\theoremstyle{plain}

\newtheorem*{theorem*}{Theorem}
\newtheorem{theorem}{Theorem}[section]

\newtheorem{lemma}[theorem]{Lemma}

\theoremstyle{remark}

\theoremstyle{Acknowledgments}

\newtheorem*{main}{Main Theorem}

\theoremstyle{definition}

\def\mod{{\rm Mod}}

\begin{document}
\blfootnote{\textup{2000} \textit{Mathematics Subject Classification}:
57N05, 20F38, 20F05}
\newenvironment{prooff}{\medskip \par \noindent {\it Proof}\ }{\hfill
$\square$ \medskip \par}
    \def\sqr#1#2{{\vcenter{\hrule height.#2pt
        \hbox{\vrule width.#2pt height#1pt \kern#1pt
            \vrule width.#2pt}\hrule height.#2pt}}}
    \def\square{\mathchoice\sqr67\sqr67\sqr{2.1}6\sqr{1.5}6}
\def\pf#1{\medskip \par \noindent {\it #1.}\ }
\def\endpf{\hfill $\square$ \medskip \par}
\def\demo#1{\medskip \par \noindent {\it #1.}\ }
\def\enddemo{\medskip \par}
\def\qed{~\hfill$\square$}

 \title[Generating the Twist Subgroup by Involutions] {Generating the Twist Subgroup by Involutions}

\author[T{\"{u}}l\.{i}n Altun{\"{o}}z,       Mehmetc\.{i}k Pamuk, and O\u{g}uz Y{\i}ld{\i}z ]{T{\"{u}}l\.{i}n Altun{\"{o}}z,    Mehmetc\.{i}k Pamuk, and O\u{g}uz Y{\i}ld{\i}z}

\address{Department of Mathematics, Middle East Technical University,
 Ankara, Turkey}
\email{atulin@metu.edu.tr}  \email{mpamuk@metu.edu.tr} \email{oguzyildiz16@gmail.com}


\begin{abstract}
For a nonorientable surface, the twist subgroup is an index $2$ subgroup of the mapping class group.  It is generated by Dehn twists about two-sided simple closed curves.
In this paper, we study involution generators of the twist subgroup.  We give generating sets of involutions with the smallest number of elements our methods allow.
\end{abstract}

 \maketitle
  \setcounter{secnumdepth}{2}
 \setcounter{section}{0}

\section{Introduction}

Let $N_g$ denote a closed connected nonorientable surface of genus $g$. The mapping class group of $N_g$ is defined to be the group of the isotopy classes of all diffeomorphisms
of $N_g$.  Throughout the paper this group will be denoted by $\mod(N_g)$.  Let $\Sigma_g$ denote a closed connected orientable surface of genus $g$. The mapping class group of 
$\Sigma_g$ is the group of the isotopy classes of orientation preserving diffeomorphisms and is denoted by $\mod(\Sigma_g)$.

In the orientable case, it is a classical result that $\mod(\Sigma_g)$ is generated by finitely many Dehn twists about nonseparating simple closed curves~\cite{de,H,l3}.  
The study of algebraic properties of mapping class group, finding small generating sets, generating sets with particular properties, is an active one leading to interesting developments. 
Wajnryb~\cite{w} showed that $\mod(\Sigma_g)$ can be generated by two elements given as a product of Dehn twists.  As the group is not abelian, this is the smallest 
possible.  Later, Korkmaz~\cite{mk2} showed that one of these generators can be taken as a Dehn twist, he also proved that $\mod(\Sigma_g)$ can be generated by two torsion 
elements. Recently, the third author showed that $\mod(\Sigma_g)$ is generated by two torsions of small orders~\cite{y1}.

Generating  $\mod(\Sigma_g)$ by involutions was first considered by McCarthy and Papadopoulus~\cite{mp}.  They showed that the group can be generated 
by infinitely many conjugates of a single involution (element of order two) for $g\geq 3$.     
In terms of generating by finitely many involutions, Luo~\cite{luo} showed that any Dehn twist about a nonseparating simple closed curve 
can be written as a product six involutions, which in turn implies that $\mod(\Sigma_g)$ can be generated by $12g+6$ involutions.  
Brendle and Farb~\cite{bf} obtained a generating set of six involutions for $g\geq3$. Following their work, Kassabov~\cite{ka} showed that 
$\mod(\Sigma_g)$ can be generated by four involutions if $g\geq7$.  Recently, Korkmaz~\cite{mk1} showed that $\mod(\Sigma_g)$ is generated by three involutions 
if $g\geq8$ and four involutions if $g\geq3$. Also, the third author improved his result showing that it is generated by three involutions if $g\geq6$~\cite{y2}.

Compared to orientable surfaces less is known about $\mod(N_g)$.  Lickorish~\cite{l1,l2} showed that it is generated by Dehn twists about two-sided simple
closed curves and a so-called $Y$-homeomorphism (or a crosscap slide). Chillingworth~\cite{c} gave a finite generating set for $\mod(N_g)$ that linearly depends on $g$. 
Szepietowski~\cite{sz2} proved that  $\mod(N_g)$ is generated by three elements and by four involutions.

The twist subgroup $\mathcal{T}_g$ of $\mod(N_g)$ is the group generated by Dehn twists about two-sided simple closed curves.
The group $\mathcal{T}_g$ is a subgroup of index $2$ in $\mod(N_g)$ ~\cite{l2}. 
Chillingworth~\cite{c} showed that $\mathcal{T}_g$ can be generated by finitely many Dehn twists. Stukow~\cite{st2} obtained a finite presentation for $\mathcal{T}_g$ with 
$(g+2)$ Dehn twist generators. Later  Omori~\cite{om} reduced the number of Dehn twist generators to $(g+1)$ for $g\geq4$. If it is not required that all generators are Dehn twists, 
Du~\cite{du} obtained a generating set consisting of three elements, two involutions and an element of order $2g$ whenever $g\geq5$ and odd. Recently, Yoshihara~\cite{yo} was interested in the problem of finding generating sets for $\mathcal{T}_g$ consisting of only involutions.
He proved that $\mathcal{T}_g$ can be generated by six involutions for $g\geq14$ and by eight involutions if $g\geq8$.

Our aim in this paper is to generate $\mathcal{T}_g$ with fewer number of involutions. It is known that any group generated by two involutions is isomorphic to a quotient of a dihedral group. Hence, $\mathcal{T}_g$ cannot be generated by two involutions. We are not sure whether $\mathcal{T}_g$ can be generated by three involutions. Based on the approach of ~\cite{mk1}, we obtain the following result:
\begin{main}\label{t0}
The twist subgroup $\mathcal{T}_g$ of $\mod(N_g)$
is generated by 
\begin{enumerate}
\item[(1)] four involutions if  $g\geq12$ and even, 
\item[(2)] four involutions if $g=4k+1\geq 5$,
\item[(3)] five involutions if $g=4k+3\geq 11$.
\end{enumerate}
We also prove that the twist subgroup $\mathcal{T}_g$ can be generated by 
\begin{enumerate}
\item[(4)] five involutions if  $g=8,10$,
\item[(5)] six involutions if $g=6,7$.
\end{enumerate}
\end{main}
Note that if a group is generated by involutions, then its first integral homology group should consist of elements of order $2$. For the twist subgroup $\mathcal{T}_g$, this is the case when $g\geq5$~\cite{st1}.

The paper is organized as follows. In Section~\ref{S2}, we recall some basic results on $\mod(N_g)$ and its subgroup $\mathcal{T}_g$.  We work with nonorientable surfaces of even genus in Section ~\ref{S3} and nonorientable surfaces of odd genus in Section ~\ref{S4}.
\medskip

\noindent
\textit{Acknowledgments.} The authors thank Mustafa Korkmaz for various fruitful discussions. The first author was partially supported by the Scientific and 
Technologic Research Council of Turkey (TUBITAK)[grant number 117F015].


\par  
\section{Background and Results on Mapping Class Groups} \label{S2}
 Let $N_g$ be a closed connected nonorientable surface of genus $g$. 
 Note that the {\textit{genus}} for a nonorientable surface is the number 
 of projective planes in a connected sum decomposition. The {\textit{mapping class group}} 
 $\mod(N_g)$ of the surface $N_g$ is defined to be the group of the isotopy classes of 
 diffeomorphisms $N_g \to N_g$. Throughout the paper we do not distinguish a 
 diffeomorphism from its isotopy class. For the composition of two diffeomorphisms, we
use the functional notation; if $g$ and $h$ are two diffeomorphisms, 
the composition $gh$ means that $h$ acts on $N_g$ first.\\
\indent
A simple closed curve on a nonorientable surface $N_g$ is said to be 
\textit{one-sided} if a regular neighbourhood of it is homeomorphic to 
a M\"{o}bius band. It is called \textit{two-sided} if a regular neighbourhood of it is homeomorphic to an annulus. If $a$ is a two-sided simple closed 
curve on $N_g$, to define the Dehn twist $t_a$, we need to fix one of two possible 
orientations on a regular neighbourhood of $a$ (as we did for the 
curve $a_1$ in Figure~\ref{G}). Following ~\cite{mk1} the right-handed 
Dehn twist $t_a$ about $a$ will be denoted by the corresponding capital 
letter $A$.

Recall the following properties of Dehn twists: let $a$ and $b$ be 
two-sided simple closed curves on $N_g$ and let $f\in \mod(N_g)$.
\begin{itemize}
\item \textbf{Commutativity:} If $a$ and $b$ are disjoint, then $AB=BA$.
\item \textbf{Conjugation:} If $f(a)=b$, then $fAf^{-1}=B^{s}$, where $s=\pm 1$ 
depending on whether $f$ is orientation preserving or orientation reversing on a 
neighbourhood of $a$ with respect to the chosen orientation.
\end{itemize}

\begin{figure}[h]
\begin{center}
\scalebox{0.3}{\includegraphics{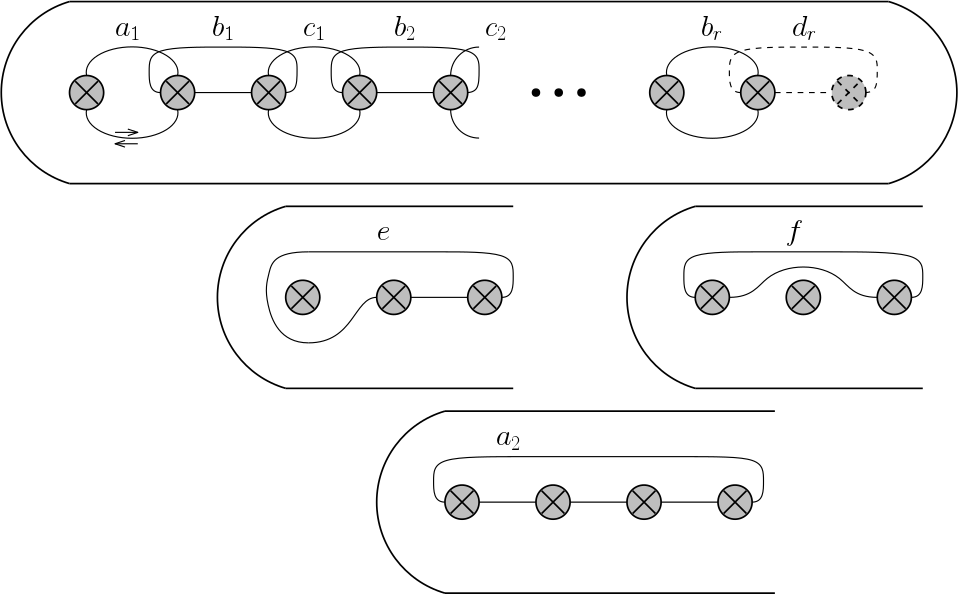}}
\caption{The curves $a_1,a_2,b_i,c_i,e$ and $f$ on the surface $N_g$.}
\label{G}
\end{center}
\end{figure}
\par
\begin{figure}[h]
\begin{center}
\scalebox{0.26}{\includegraphics{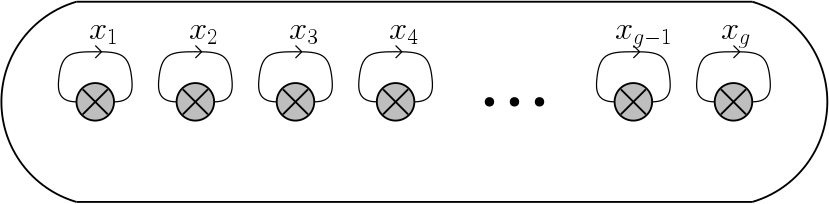}}
\caption{Generators of $H_1(N_g;\mathbb{R})$.}
\label{H}
\end{center}
\end{figure}

Consider the surface $N_g$ shown in Figure~\ref{G}.
The Dehn twist generators of Omori can be given as follows (note that we do not have the curve $d_r$ when $g$ is odd).
\begin{theorem}\cite{om}\label{thm1}
The twist subgroup $\mathcal{T}_g$ is generated by the following $(g+1)$ Dehn twists
\begin{enumerate}
\item $A_1,A_2,B_1,\ldots, B_r$, $C_1,\ldots, C_{r-1}$ and $E$ if $g=2r+1$ and 
\item $A_1,A_2,B_1,\ldots, B_r$, $C_1,\ldots, C_{r-1}$, $D_r$ and $E$ if $g=2r+2$ .
\end{enumerate}
\end{theorem}

Consider a basis $\lbrace x_1, x_2. \ldots, x_{g-1}\rbrace$ for $H_1(N_g; \mathbb{R})$ 
such that the curves $x_i$ are one-sided and disjoint as in Figure~\ref{H}. It is known that every
 diffeomorphism $f: N_g \to N_g$ induces a linear map 
 $f_{\ast}: H_1(N_g;\mathbb{R}) \to H_1(N_g;\mathbb{R})$. Therefore, one can
  define a homomorphism $D: \mod(N_g) \to \mathbb{Z}_{2}$ by $D(f)=\textrm{det}(f_{\ast})$. 
  The following lemma from~\cite{l1} tells when a mapping class falls into the twist subgroup $\mathcal{T}_g$.

\begin{lemma}\label{lem1} Let $f\in  \mod(N_g)$. Then  $D(f)=1$ if $f\in \mathcal{T}_g$ and
$D(f)=-1$ if $f \not \in \mathcal{T}_g$.
\end{lemma}

\section{The even case}\label{S3}
\begin{figure}[h]
\begin{center}
\scalebox{0.27}{\includegraphics{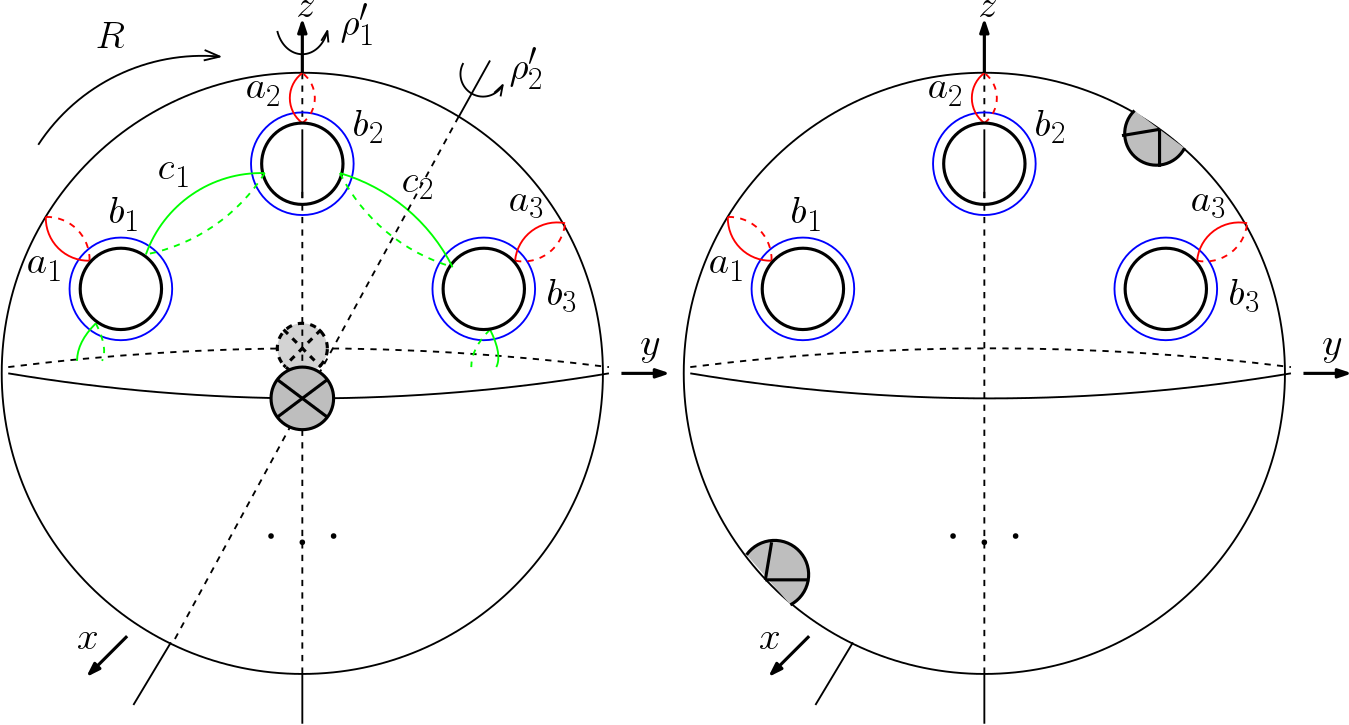}}
\caption{The models for $N_g$ if $g=2r+2$.}
\label{C}
\end{center}
\end{figure}

For $g=2r+2$, we work with the models in Figure~\ref{C}. This surface is obtained
 from a genus $r$ orientable surface by deleting the interiors of two disjoint disks 
and identifying the antipodal points on the boundary. Moreover, the genus $r$ 
surface minus two disks is embedded in $\mathbb{R}^{3}$ in such a way that each genus 
is in a circular position with the second genus on the $+z$-axis and the rotation $R$ by
$\frac{2\pi}{r}$ about $x$-axis maps the curve $b_i$ to $b_{i+1}$ for
$i=1,\ldots,r-1$ and $b_r$ to $b_1$.\par
We use the explicit homeomorphism constructed in  ~\cite[Section $3$]{st1} to identify the models in Figure~\ref{G} and~\ref{C}. On the left hand side of Figure~\ref{C},
one of the crosscaps is centered on the $+x$-axis and the other one is obtained by 
rotating the first one by $\pi$ about the $z$-axis.  The model on the right hand side is obtained 
from the model on the left hand side by sliding crosscaps via a diffeomorphism, say $\phi$.

Let $\tau$ be the blackboard reflection of $N_g$ for the model in 
the left hand side in Figure~\ref{C}.  If $r$ is odd, we consider the reflection 
$\tau$ and if $r$ is even, we consider the reflection $\phi\tau\phi^{-1}$. 
The surface $N_g$ is invariant under the reflections $\tau$ and $\phi\tau\phi^{-1}$. 
Abusing the notation, we keep writing $\tau$ instead of $\phi\tau\phi^{-1}$. Note that $D(\tau)=-1$.\\
\indent
Note that the surface $N_g$ is invariant under the two rotations $\rho_{1}^{\prime}$ and $\rho_{2}^{\prime}$ 
where $\rho_{1}^{\prime}$ is the rotation by $\pi$ about $z$-axis and $\rho_{2}^{\prime}$ is the rotation by $\pi$ 
about the line $z=tan(\frac{\pi}{r})y$, $x=0$ as in Figure~\ref{C}. The rotations $\rho_{1}^{\prime}$ and $\rho_{2}^{\prime}$ 
satisfy $D(\rho_{1}^{\prime})=D(\rho_{2}^{\prime})=-1$, which implies that the twist subgroup $\mathcal{T}_g$ does not 
contain $\rho_{1}^{\prime}$ and $\rho_{2}^{\prime}$. 
Let $\rho_1=\rho_{1}^{\prime}\tau$ and $\rho_2=\rho_{2}^{\prime}\tau$. Then the involutions $\rho_{1}$ and $\rho_{2}$ are contained 
in $\mathcal{T}_g$ by Lemma~\ref{lem1}. Observe that the rotation $R=\rho_2\rho_1$.
 
\subsection{Generating sets for the twist subgroup  $\mathcal{T}_g$} 
Recently, Korkmaz~\cite{mk1} introduced new generating sets for the mapping class group 
of an orientable surface. We follow the outline of his proofs. Especially, since the curves 
$a_i$, $b_i$ and $c_i$ are exactly the same as in ~\cite{mk1}, statements about these 
curves follows directly from ~\cite{mk1}. Before we state our result, let us recall the 
above mentioned theorem of Korkmaz. Recall that $A_i$, $B_i$, $C_i$, $E$
and $F$ represent the Dehn twists about the corresponding lower case letters in Figure~\ref{G} and ~\ref{C}.

\begin{theorem}\cite{mk1}\label{mt1}
Let $\Sigma_g$ denote a closed connected oriented surface of genus $g$. Then, 
if $g\geq3$, $\mod(\Sigma_g)$ is generated by the four elements $R,
A_1A_{2}^{-1}, B_1B_{2}^{-1}$ and $C_1C_{2}^{-1}$.
\end{theorem}
Using the above theorem, we give a generating set for $\mathcal{T}_g$ when $g$ is even.
\begin{theorem}\label{t1}
Let $r\geq3$ and $g=2r+2$. Then the twist subgroup $\mathcal{T}_g$ is generated by 
the elements $R, A_1A_{2}^{-1}, B_1B_{2}^{-1}, C_1C_{2}^{-1}, D_r$ and $E$ if $g=2r+2$.

\end{theorem}
\begin{proof}
Let $G$ be the subgroup of $\mathcal{T}_g$  generated by the set 
\[
\lbrace R, A_1A_{2}^{-1}, B_1B_{2}^{-1}, C_1C_{2}^{-1}, D_r, E\rbrace
\]
if $g=2r+2$.

Let $\mathcal{S}$ denote the set of isotopy classes of two-sided non-separating 
simple closed curves on $N_g$. Define a subset $\mathcal{G}$ of $\mathcal{S}\times \mathcal{S}$ 
as 
\[
\mathcal{G} =\lbrace(a,b): AB^{-1}\in G \rbrace.
\]
The set $\mathcal{G}$ defines an equivalence relation on $\mathcal{S}$ which satisfies 
$G$-invariance property, that is, 
\begin{center}
if $(a,b)\in \mathcal{G}$ and $H\in G$ then $(H(a),H(b))\in \mathcal{G}$.
\end{center}
Then it follows from the proof of Theorem~\ref{mt1} that the Dehn twists 
$A_i$ and $B_i$ for $i=1,\ldots,r$ are contained in $G$. Also, $G$ contains $C_j$ 
for $j=1,\ldots,r-1$. Since all generators given in Theorem~\ref{thm1} are contained 
in the group $G$. We conclude that $G=\mathcal{T}_g$.
\end{proof}

\subsection{Involution generators}
We consider the surface $N_g$ where $g$-crosscaps are 
distributed on the sphere as in Figure~\ref{sigma}. If $g=2r+2$ and $r\geq3$, there is a 
reflection, $\sigma$, of the surface $N_g$ in the $xy$-plane such that 
\begin{itemize}
	\item $\sigma(f)=a_1$, $\sigma(b_r)=d_r$,
	\item $\sigma(x_2)=x_3$, $\sigma(x_4)=x_5$ $\sigma(x_{g-2})=x_{g}$ and
	\item $\sigma(x_i)=x_i$ if $i=6,\ldots,g-3$ or $i=1,g-1$.
\end{itemize}
with reverse orientation. (Recall that $x_i$'s are the generators of $H_1(N_g;\mathbb{R})$ 
as shown in Figure~\ref{H}.)\\
\begin{figure}[h]
\begin{center}
\scalebox{0.35}{\includegraphics{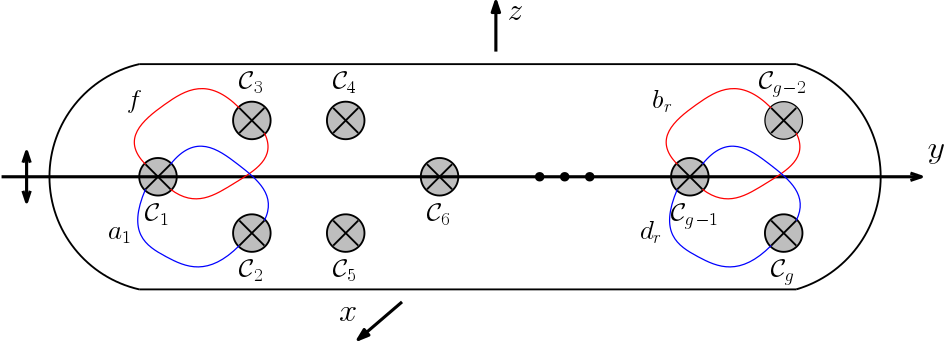}}
\caption{The involution $\sigma$ if $g=2r+2$.}
\label{sigma}
\end{center}
\end{figure}
\indent
The linear map $D$ associated to $\sigma$ satisfies $D(\sigma)=1$ if $g$ is even. 
This implies that the involution $\sigma$ is contained in $\mathcal{T}_g$ if $g$ is even. 
 
\indent
\begin{figure}[h]
\begin{center}
\scalebox{0.4}{\includegraphics{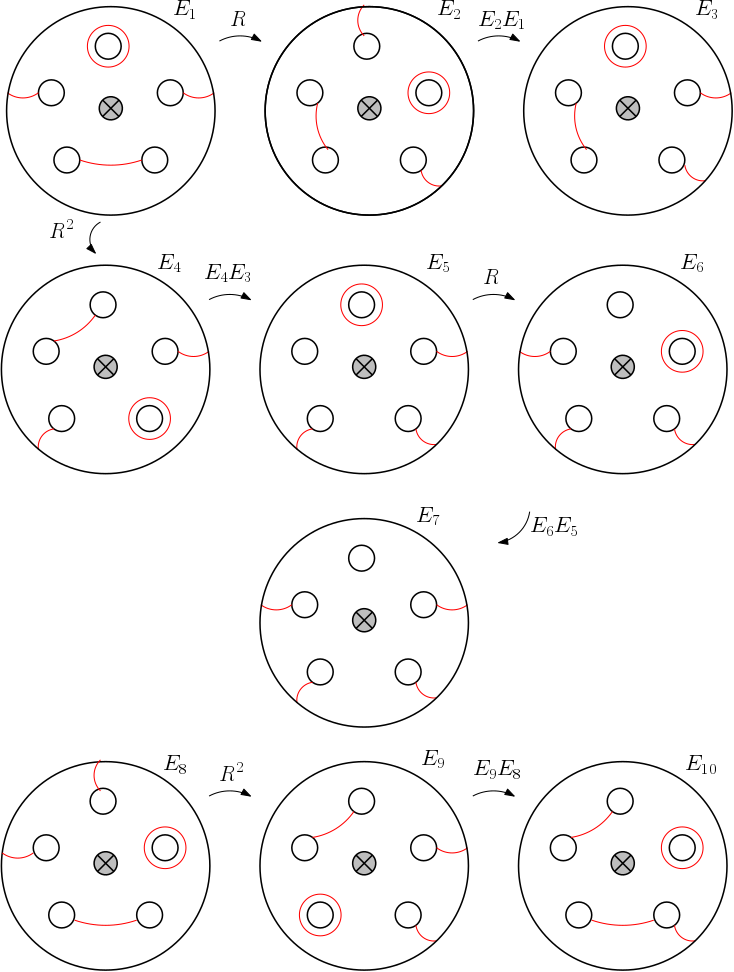}}
\caption{The proof of Theorem~\ref{g12}.}
\label{E}
\end{center}
\end{figure}
\begin{theorem}\label{g12}
The twist subgroup $\mathcal{T}_{12}$ is generated by the involutions
$\rho_1,\rho_2, \rho_1 A_1 B_2 C_4 A_3$ and $\sigma$.
\end{theorem}
\begin{proof}
Consider the surface $N_{12}$ as in Figure~\ref{C}. Since 
\[
 \rho_1(a_1)=a_3, \rho_1(b_2)=b_2 \textrm{ and }
	\rho_1(c_4)=c_4,
\]
and $\tau$ reverses the orientation of a neighbourhood of a two-sided simple closed curve, we get
\begin{itemize}
	\item $\rho_1A_1\rho_1=A_{3}^{-1}$,
	\item $\rho_1B_2\rho_1=B_{2}^{-1}$ and
	\item $\rho_1C_4\rho_1=C_{4}^{-1}$.
	\end{itemize}
	 It is easy to verify that 
	$\rho_1 A_1 B_2 C_4 A_3$ is an involution. Let $E_1=A_1 B_2 C_4 A_3$
	and let $H$ be the subgroup of $\mathcal{T}_{12}$ generated by the set
	\[
	\lbrace \rho_1,\rho_2, \rho_1E_1, \sigma\rbrace.
	\]
Note that the rotation $R$ is in the subgroup $H$. By Theorem~\ref{t1}, we need to show 
that the elements $A_1A_{2}^{-1}, B_1B_{2}^{-1}, C_1C_{2}^{-1},D_5$ and $E$ are 
contained in $H$.\\
\noindent
Let $E_2=RE_1R^{-1}=A_2B_3C_5A_4$.
It can be easily shown that
\[
E_2E_1(a_2,b_3,c_5,a_4)=(b_2,a_3,c_5,a_4),
\]
so that $E_3=B_2A_3C_5A_4$	is in $H$.\\
\noindent
Let 
\[
E_4=R^{2}E_1R^{-2}=A_3B_4C_1A_5.
\]
It is easy to show that
\[
E_4E_3(a_3,b_4,c_1,a_5)=(a_3,a_4,b_2,a_5)
\]
so that $E_5=A_3A_4B_2A_5$ are contained in $H$. Hence,
\[
E_5E_{3}^{-1}=A_5C_{5}^{-1}\in H.
\]
One can easily see that the elements $A_iC_{i}^{-1}$ are contained in $H$ by 
conjugating $A_5C_{5}^{-1}$ with powers of $R$.\\
\noindent
Let 
\[
E_6=RE_5R^{-1}=A_4A_5B_3A_1. 
\]
One can easily show that 
\[
E_6E_5(a_4,a_5,b_3,a_1)=(a_4,a_5,a_3,a_1)
\]
so that $E_7=A_4A_5A_3A_1$ is in $H$. Therefore,
\[
E_7E_{6}^{-1}=A_3B_{3}^{-1}\in H.
\]
By conjugating with powers of $R$, we get $A_iB_{i}^{-1}\in H$. Hence,
\[
B_5C_{5}^{-1}=(B_5A_{5}^{-1})(A_5C_{5}^{-1})\in H.
\]
Again by conjugating with powers of $R$, the elements $B_iC_{i}^{-1}$ 
are contained in $H$.\\
\noindent
Let
\[
E_8=(A_2B_{2}^{-1})(B_3A_{3}^{-1})E_1=A_1A_2C_4B_3
\]
and 
\[
E_9=R^{2}F_8R^{-2}=A_3A_4C_1B_5.
\]
It can also be shown that 
\[
E_9E_8(a_3,a_4,c_1,b_5)=(b_3,a_4,c_1,c_4)
\]
so that $E_{10}=B_3A_4C_1C_4$. Hence,
\[
E_9E_{10}^{-1}B_{3}A_{3}^{-1}=B_5C_{4}^{-1}\in H.
\]
The conjugation of this with powers of $R$ implies that $B_{i+1}C_{i}^{-1}\in H$. Hence
\begin{itemize}
	\item $A_1A_{2}^{-1}=(A_1C_{1}^{-1})(C_1B_{2}^{-1})(B_2A_{2}^{-1})$,
	\item $B_1B_{2}^{-1}=(B_1C_{1}^{-1})(C_1B_{2}^{-1})$ and
	\item $C_1C_{2}^{-1}=(C_1B_{2}^{-1})(B_2C_{2}^{-1})$
	\end{itemize}
are contained in $H$. Also it follows from the fact that 
\[
\sigma(a_1)=f \text{ and } \sigma(b_5)=d_5
\]
with a choice of orientations of regular neighbourhoods of the curves, the element 
$D_5$ and $F$ are contained in $H$. By the fact that $A_1(f)=e$, $E$ is in $H$. 
We conclude that $H=T_{12}$.
\end{proof}
\
\begin{theorem}\label{rodd}
For $g=2r+2$, the twist subgroup $\mathcal{T}_{g}$ is generated by the 
involutions $\rho_1,\rho_2,\rho_1A_1B_2C_{\frac{r+3}{2}} A_3$ and $\sigma$ if $r\geq7$ 
and odd.
\end{theorem}
\begin{proof}
Consider the surface $N_{g}$ as in Figure~\ref{C}. We have
\[
 \rho_1(a_1)=a_3, \rho_1(b_2)=b_2 \textrm{ and } \rho_1(c_{\frac{r+3}{2}})=c_{\frac{r+3}{2}}.
\]
Since $\tau$ reverses the orientation of a neighbourhood of a two-sided simple closed curve, we get
\begin{itemize}
	\item $\rho_1A_1\rho_1=A_{3}^{-1}$ 
	\item $\rho_1B_2\rho_1=B_{2}^{-1}$ and
	\item $\rho_1C_{\frac{r+3}{2}}\rho_1=C_{\frac{r+3}{2}}^{-1}$.
	\end{itemize}
 It can be shown that 
	$\rho_1A_1B_2C_{\frac{r+3}{2}}A_3$ is an involution. 
Let $G_1=A_1B_2C_{\frac{r+3}{2}}A_3$  and let $K$ be the subgroup of $\mathcal{T}_{g}$ generated by the set
	\[
	\lbrace \rho_1,\rho_2, \rho_1G_1, \sigma\rbrace .
	\]
Note that the rotation $R$ is in $K$. By Theorem	~\ref{t1}, we need to 
show that the elements $A_1A_{2}^{-1}, B_1B_{2}^{-1}, C_1C_{2}^{-1},D_r$ and $E$ 
are contained in $K$.\\
\noindent
It follows from 
\begin{itemize}
\item$G_2=RG_1R^{-1}=A_2B_3C_{\frac{r+5}{2}}A_4\in K$,
\item $G_3=(G_2G_1)G_2(G_2G_1)^{-1}=B_2A_3C_{\frac{r+5}{2}}A_4\in K$,
\item $G_4=RG_3R^{-1}=B_3A_4C_{\frac{r+7}{2}}A_5\in K$,
\item $G_5=(G_4G_3)G_4(G_4G_3)^{-1}=A_3A_4C_{\frac{r+7}{2}}A_5\in K$
\end{itemize}
that 
\[
G_4G_{5}^{-1}=B_3A_{3}^{-1}\in K.
\]
Hence, the elements $B_iA_{i}^{-1}$ are contained in $K$ by conjugating $B_3A_{3}^{-1}$
with powers of $R$. Let
\begin{itemize}
\item $G_6=R^{\frac{r-3}{2}}G_4R^{\frac{3-r}{2}}=B_{\frac{r+3}{2}}A_{\frac{r+5}{2}}C_2A_{\frac{r+7}{2}} \in K$,
\item $G_7=(G_6G_4)G_6(G_6G_4)^{-1}=B_{\frac{r+3}{2}}A_{\frac{r+5}{2}}B_3A_{\frac{r+7}{2}} \in K$ if $r>7$,\\
($G_7=A_5A_6B_3A_7 \in K$ if $g=7$).
\end{itemize}
Then 
\[
G_7G_{6}^{-1}=B_3C_{2}^{-1} \in K \textrm{ if } r>7
\]
and
\[
G_7G_{6}^{-1}B_5A_{5}^{-1}=B_3C_{2}^{-1} \in K \textrm{ if } r=7.
\]
Therefore, the elements $B_{i+1}C_{i}^{-1}$ are contained in the group $K$ by 
conjugating $B_3C_{2}^{-1}$ with powers of $R$. Let 
 \begin{itemize}
 \item $G_8=R^{\frac{r-1}{2}}G_4R^{\frac{1-r}{2}}=B_{\frac{r+5}{2}}A_{\frac{r+7}{2}}C_3A_{\frac{r+9}{2}} \in K$ if $r>7$,\\
 ($G_8=B_6A_7C_3A_1 \in K$ if $r=7$),
 \item $G_{9}=(G_8G_4)G_8(G_8G_4)^{-1}=B_{\frac{r+5}{2}}A_{\frac{r+7}{2}}B_3A_{\frac{r+9}{2}} \in K$ if $r>7$.\\
 ($G_9=B_6A_7B_3A_1 \in K$ if $r=7$).
 \end{itemize}
 Then 
 \[
 G_9G_{8}^{-1}=B_3C_3^{-1} \in K
\textrm{ if } r\geq7.
\]
This implies that the subgroup $K$ contains $B_{i}C_{i}^{-1}$ by 
conjugating $B_3C_{3}^{-1}$ with powers of $R$.
The rest of the proof is very similar to the proof of Theorem~\ref{g12}. 
\end{proof}
\begin{theorem}\label{t3.5}
For $g=2r+2$, the twist subgroup $\mathcal{T}_{g}$ is generated by 
the involutions $\rho_1,\rho_2,\rho_1A_2C_{\frac{r}{2}}B_{\frac{r+4}{2}} C_{\frac{r+6}{2}}$ and $\sigma$ if $r\geq6$ and even.
\end{theorem}
\begin{proof}
Consider the surface $N_{g}$ as in Figure~\ref{C}. The involution $\rho_1$ satisfies
\[
 \rho_1(a_2)=a_2, \rho_1(b_{\frac{r+4}{2}})=b_{\frac{r+4}{2}} \textrm{ and } \rho_1(c_{\frac{r}{2}})=c_{\frac{r+6}{2}}.
\]
Since $\tau$ reverses the orientation of a neighbourhood of a two-sided simple closed curve, we have
\begin{itemize}
	\item $\rho_1A_2\rho_1=A_{2}^{-1}$ 
	\item $\rho_1B_{\frac{r+4}{2}}\rho_1=B_{\frac{r+4}{2}}^{-1}$ and
	\item $\rho_1C_{\frac{r}{2}}\rho_1=C_{\frac{r+6}{2}}^{-1}$.
	\end{itemize}
 It can be 
shown that $\rho_1A_2C_{\frac{r}{2}}B_{\frac{r+4}{2}} C_{\frac{r+6}{2}}$ 
is an involution. Let $H_1=A_2C_{\frac{r}{2}}B_{\frac{r+4}{2}} C_{\frac{r+6}{2}}$ 
and let $K$ be the subgroup of $\mathcal{T}_{g}$ generated by the set
	\[
	\lbrace \rho_1,\rho_2, \rho_1H_1, \sigma\rbrace .
	\]
Note that the rotation $R$ is in $K$. By Theorem	~\ref{t1}, we need to show that 
the elements $A_1A_{2}^{-1}, B_1B_{2}^{-1}, C_1C_{2}^{-1},D_r$ and $E$ are contained in $K$.\\
Let
\begin{itemize}
\item$H_2=RH_1R^{-1}=A_3C_{\frac{r+2}{2}}B_{\frac{r+6}{2}}C_{\frac{r+8}{2}} \in K$,
\item $H_3=(H_2H_1)H_2(H_2H_1)^{-1}=A_3B_{\frac{r+4}{2}}C_{\frac{r+6}{2}}C_{\frac{r+8}{2}}\in K$,
\item $H_4=RH_3R^{-1}=A_4B_{\frac{r+6}{2}}C_{\frac{r+8}{2}}C_{\frac{r+10}{2}}\in K$,
\item $H_5=(H_4H_3)H_4(H_4H_3)^{-1}=A_4C_{\frac{r+6}{2}}C_{\frac{r+8}{2}}C_{\frac{r+10}{2}}\in K$.
\end{itemize}
Then, we get
\[
H_4H_5^{-1}=B_{\frac{r+6}{2}}C_{\frac{r+6}{2}}^{-1} \in K
\]
and 
\[
H_2H_3^{-1}\Big(C_{\frac{r+6}{2}}B_{\frac{r+6}{2}}^{-1}\Big)=C_{\frac{r+2}{2}}B_{\frac{r+4}{2}}^{-1}\in K.
\]
By conjugating the elements $B_{\frac{r+6}{2}}C_{\frac{r+6}{2}}^{-1}$ and $C_{\frac{r+2}{2}}B_{\frac{r+4}{2}}^{-1}$ with powers of $R$, we conclude that $B_iC_{i}^{-1}$  and $C_iB_{i+1}^{-1}$ are contained in $K$.\\
\noindent
Let
\begin{itemize}
\item $H_6=(B_{\frac{r+6}{2}}C_{\frac{r+6}{2}}^{-1})(B_{\frac{r}{2}}C_{\frac{r}{2}}^{-1})H_1=B_{\frac{r+6}{2}}B_{\frac{r}{2}}A_2B_{\frac{r+4}{2}} \in K$,
\item $H_7=R^{\frac{r-4}{2}}H_6R^{\frac{4-r}{2}}=A_{\frac{r}{2}}B_{r-2}B_rB_1 \in K$,
\item $H_8=(H_7H_6)H_7(H_7H_6)^{-1}=B_{\frac{r}{2}}B_{r-2}B_rB_1 \in K$.
\end{itemize}
Then
\[
H_8H_7^{-1}=B_{\frac{r}{2}}A_{\frac{r}{2}}^{-1} \in K.
\]
By conjugating with powers of $R$, $K$ contains $B_iA_{i}^{-1}$. The rest of the proof is very similar to the proof of Theorem~\ref{g12}. 
\end{proof}
In the rest of this section, we introduce involution generators for $\mathcal{T}_g$ for $g=6,8$ and $10$.\\
\begin{figure}[h]
\begin{center}
\scalebox{0.35}{\includegraphics{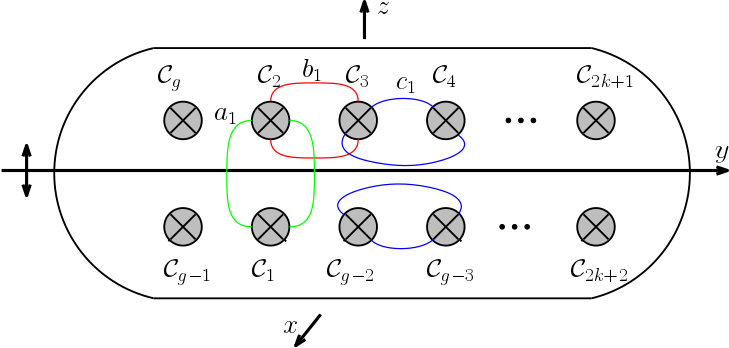}}
\caption{The involution $\delta_1$ for $g=4k+2$.}
\label{D1}
\end{center}
\end{figure}
\begin{figure}[h]
\begin{center}
\scalebox{0.35}{\includegraphics{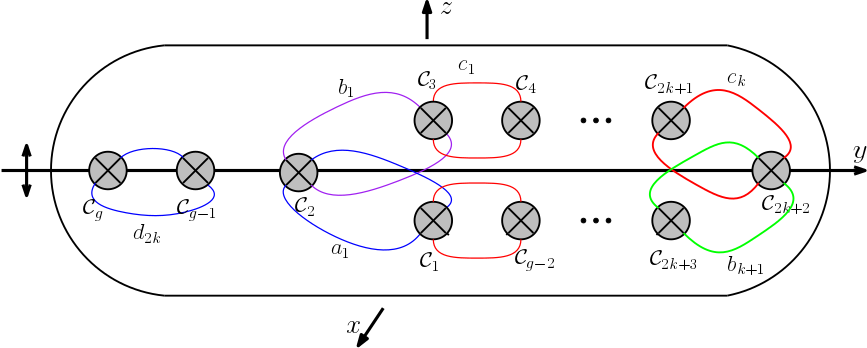}}
\caption{The involution $\delta_2$ for $g=4k+2$.}
\label{D2}
\end{center}
\end{figure}
\begin{figure}[h]
\begin{center}
\scalebox{0.35}{\includegraphics{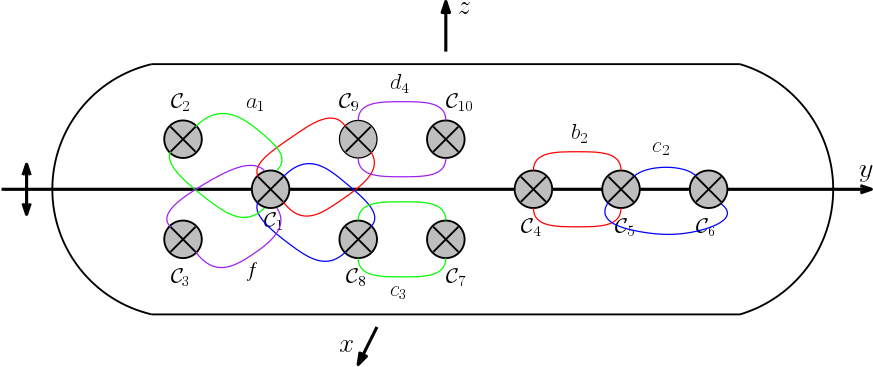}}
\caption{The involution $\delta_3$ for $g=10$.}
\label{D3}
\end{center}
\end{figure}
We consider the models for the surface $N_{10}$, where $10$-crosscaps are 
distributed on the sphere as in Figure~\ref{D1}, ~\ref{D2} and ~\ref{D3}.  There are 
reflections, $\delta_1,\delta_2$ and $\delta_3$, of the surface $N_{10}$ in the $xy$-plane such that 
\begin{itemize}
	\item $\delta_1(x_i)=x_{i+1}$ if $i=1,5,9$, \\
	$\delta_1(x_3)=x_{8}$, $\delta_1(x_4)=x_{7}$,
	\item $\delta_2(x_i)=x_{i}$ if $i=2,6,9,10$,\\
	 $\delta_2(x_1)=x_{3}$, $\delta_2(x_4)=x_{8}$,  $\delta_2(x_5)=x_{7}$ and
	\item $\delta_3(x_i)=x_{i}$ if $i=1,4,5,6$,\\
	 $\delta_3(x_2)=x_{3}$, $\delta_3(x_8)=x_{9}$,  $\delta_3(x_7)=x_{10}$.
\end{itemize}
 Recall that $x_i$'s are the generators of $H_1(N_g;\mathbb{R})$ 
as shown in Figure~\ref{H}. Note that the involutions $\delta_1,\delta_2$ and $\delta_3$ reverse the orientation of a neighbourhood of a two-sided simple closed curve. Since $D(\delta_i)=1$, the involutions $\delta_i$ are in $\mathcal{T}_{10}$ for $i=1,2,3$.
\begin{theorem}
The twist subgroup $\mathcal{T}_{10}$ is generated by five involutions $\delta_1,\delta_2, \delta_2\delta_1\delta_2A_2, \delta_1A_1,\delta_3$. 
\end{theorem}
\begin{proof}
Let $K$ be the subgroup of $\mathcal{T}_{10}$ generated by the set
\[
\lbrace \delta_1,\delta_2, \delta_2\delta_1\delta_2A_2, \delta_1A_1,\delta_3 \rbrace.
\] 
It is clear that $ \delta_2\delta_1\delta_2A_2$ and  $\delta_1A_1$ are involutions.
It follows from 
\begin{itemize}
	\item $A_1=\delta_1(\delta_1A_1)$ and
	\item $A_2=( \delta_2\delta_1\delta_2)(\delta_2\delta_1\delta_2A_2)$
\end{itemize}
that the elements $A_1$ and $A_2$ are in $K$. Also,
It follows from 
\begin{itemize}
	\item $\delta_2(a_1)=b_1$ 
	\item $\delta_2\delta_1(b_i)=c_i$ for $i=1,2,3$ and
	\item $\delta_2\delta_1(c_i)=b_{i+1}$ for $i=1,2$
\end{itemize}
that $B_i,C_i$ are contained in $K$ for $i=1,2,3$. Moreover,
since
\begin{itemize}
	\item $\delta_3(c_3)=d_4$, 
	\item $\delta_1\delta_2\delta_3\delta_1\delta_3(c_1)=b_4$ and
	\item $A_1\delta_3(a_1)=e$
\end{itemize}
then the elements $D_4, B_4$
and $E$ are in $K$. We conclude that $K=\mathcal{T}_{10}$ by Theorem~\ref{thm1}.
\end{proof}
We consider the models for the surface $N_{8}$, where $8$-crosscaps are 
distributed on the sphere as in Figure~\ref{L1}, ~\ref{L2} and ~\ref{L3}.  There are 
reflections, $\lambda_1,\lambda_2$ and $\lambda_3$, of the surface $N_{8}$ in the $xy$-plane such that 
\begin{itemize}
	\item $\lambda_1(x_i)=x_{i}$ if $i=7,8$, \\
	$\lambda_1(x_i)=x_{i+1}$ if $i=1,4$ and
	 $\lambda_1(x_3)=x_{6}$,
	\item $\lambda_2(x_i)=x_{i}$ if $i=2,5$,\\
	 $\lambda_2(x_1)=x_{3}$, $\lambda_2(x_4)=x_{6}$,  $\lambda_2(x_7)=x_{8}$, and
	\item $\lambda_3(x_i)=x_{i}$ if $i=1,4$,\\
	 $\lambda_3(x_2)=x_{3}$, $\lambda_3(x_5)=x_{8}$ and $\lambda_3(x_6)=x_{7}$.
\end{itemize}
  Note that the involutions $\lambda_i$ reverse the orientation of a neighbourhood of a two-sided simple closed curve for $i=1,2,3$. Since $D(\delta_i)=1$, the involutions $\delta_i$ are contained in $\mathcal{T}_{8}$ for $i=1,2,3$.\\

\begin{figure}[h]
\begin{center}
\scalebox{0.4}{\includegraphics{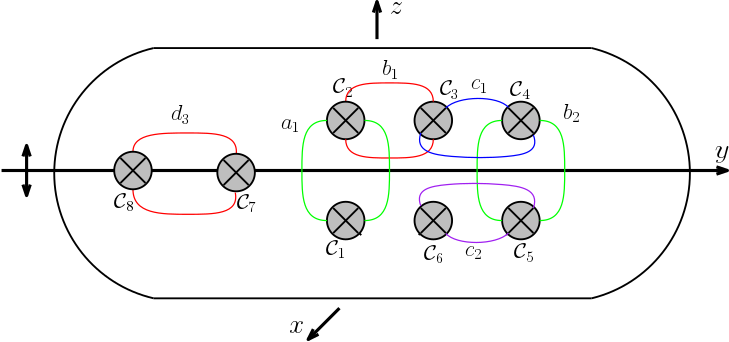}}
\caption{The involution $\lambda_1$ for $g=8$.}
\label{L1}
\end{center}
\end{figure}
\begin{figure}[h]
\begin{center}
\scalebox{0.4}{\includegraphics{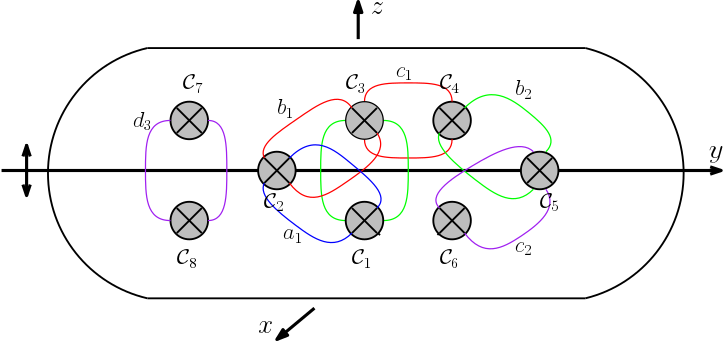}}
\caption{The involution $\lambda_2$ for $g=8$.}
\label{L2}
\end{center}
\end{figure}
\begin{figure}[h]
\begin{center}
\scalebox{0.4}{\includegraphics{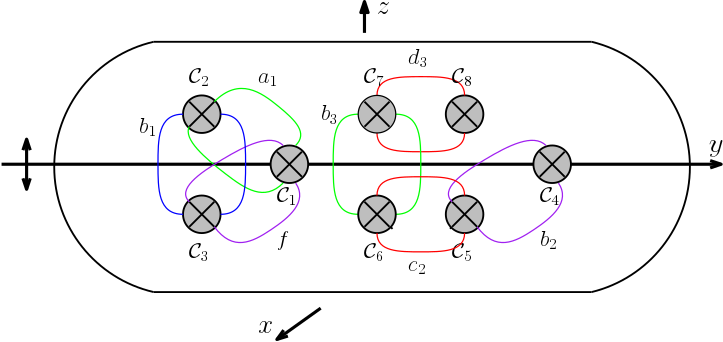}}
\caption{The involution $\lambda_3$ for $g=8$.}
\label{L3}
\end{center}
\end{figure}
\begin{theorem}
The twist subgroup $\mathcal{T}_{8}$ is generated by five involutions $\lambda_1,\lambda_2, \lambda_2\lambda_1\lambda_2A_2, \lambda_1A_1$ and $\lambda_3$. 
\end{theorem}
\begin{proof}
Let $K$ be the subgroup of $\mathcal{T}_{8}$ generated by the set
\[
\lbrace \lambda_1,\lambda_2, \lambda_2\lambda_1\lambda_2A_2, \lambda_1A_1,\lambda_3 \rbrace.
\] 
It is clear that $ \lambda_2\lambda_1\lambda_2A_2$ and  $\lambda_1A_1$ are involutions.
It follows from 
\begin{itemize}
	\item $A_1=\lambda_1(\lambda_1A_1)$ and
	\item $A_2=( \lambda_2\lambda_1\lambda_2)(\lambda_2\lambda_1\lambda_2A_2)$
\end{itemize}
that the elements $A_1$ and $A_2$ are in $K$. Also,
It follows from 
\begin{itemize}
	\item $\lambda_2\lambda_1(a_1)=b_1$, 
	\item $\lambda_2\lambda_1(b_i)=c_i$ for $i=1,2$, 
	\item $\lambda_2\lambda_1(c_1)=b_{2}$,
   \item $\lambda_3(c_2)=d_3$, 
	\item $\lambda_1\lambda_2\lambda_3\lambda_1\lambda_3(c_1)=b_3$ and
	\item $A_1\lambda_3(a_1)=e$
\end{itemize}
that all generators of $\mathcal{T}_{8}$ given in Theorem~\ref{thm1} are contained in $K$. This completes the proof.
\end{proof}
\begin{figure}[h]
\begin{center}
\scalebox{0.38}{\includegraphics{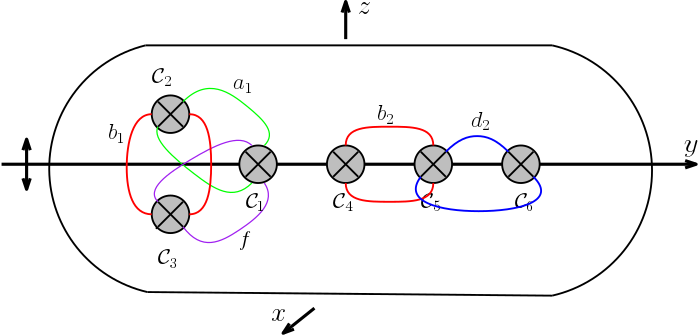}}
\caption{The involution $\xi_1$ for $g=6$.}
\label{X1}
\end{center}
\end{figure}
\begin{figure}[h]
\begin{center}
\scalebox{0.38}{\includegraphics{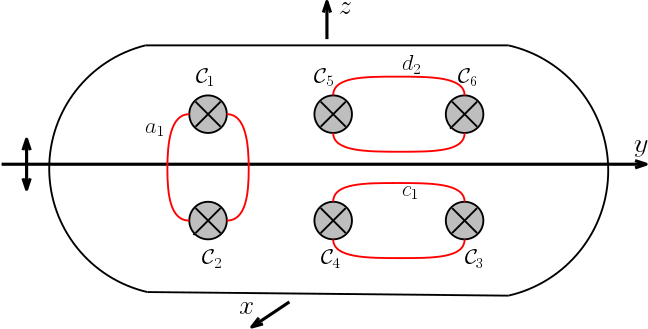}}
\caption{The involution $\xi_2$ for $g=6$.}
\label{X2}
\end{center}
\end{figure}

We consider the models for the surface $N_{6}$, where $6$-crosscaps are 
distributed on the sphere as in Figure~\ref{D1},~\ref{D2},~\ref{X1} and ~\ref{X2}. There are 
reflections $\delta_1,\delta_2, \xi_1$ and $\xi_2$ such that
\begin{itemize}
	\item $\delta_1(x_i)=x_{i+1}$ if $i=1,3,5$, 
	\item $\delta_2(x_i)=x_{i}$ if $i\neq1,3$ and
	 $\delta_2(x_1)=x_{3}$,
	\item $\xi_1(x_i)=x_{i}$ if $i\neq2,3$ and $\xi_1(x_2)=x_3$ and 
	\item $\xi_2(x_i)=x_{i+1}$ if $i=1,4$ and $\xi_2(x_3)=x_6$.
\end{itemize}
Note that the involutions $\delta_i$ and $\xi_i$ reverse the orientation of a neighbourhood of a two-sided simple closed curve for $i=1,2$. We obtain that $D(\delta_i)=D(\xi_i)=1$,  the twist subgroup $\mathcal{T}_{8}$ contains the involutions $\delta_i$ and $\xi_i$ for $i=1,2$.
\begin{theorem}
The twist subgroup $\mathcal{T}_{6}$ is generated by six involutions $\delta_1,\delta_2, \delta_2\delta_1\delta_2A_2, \delta_1A_1,\xi_1$ and $\xi_2$. 
\end{theorem}
\begin{proof}
Let $K$ be the subgroup of $\mathcal{T}_{6}$ generated by the set
\[
\lbrace \delta_1,\delta_2, \delta_2\delta_1\delta_2A_2, \delta_1A_1,\xi_1,\xi_2 \rbrace.
\] 
It is clear that $ \delta_2\delta_1\delta_2A_2$ and  $\delta_1A_1$ are involutions.
It follows from 
\begin{itemize}
	\item $A_1=\delta_1(\delta_1A_1)$ and
	\item $A_2=( \delta_2\delta_1\delta_2)(\delta_2\delta_1\delta_2A_2)$
\end{itemize}
that the elements $A_1$ and $A_2$ are in $K$. Also,
It follows from 
\begin{itemize}
	\item $\delta_2(a_1)=b_1$, 
	\item $\delta_2\delta_1(b_1)=c_1$, 
	\item $\delta_1\delta_2\xi_2(b_1)=b_2$,
	\item$\xi_2(c_1)=d_2$ and
	\item $A_1\xi_1(a_1)=e$
\end{itemize}
that all generators of $\mathcal{T}_{6}$ given in Theorem~\ref{thm1} are contained in $K$. This completes the proof.
\end{proof}

\section{The odd case}\label{S4}
For $g=4k+1$, we work with two models for $N_g$: one is on the left hand side of Figure~\ref{SO}, the other one is depicted in Figure~\ref{CODD}. 
\begin{figure}[h]
\begin{center}
\scalebox{0.4}{\includegraphics{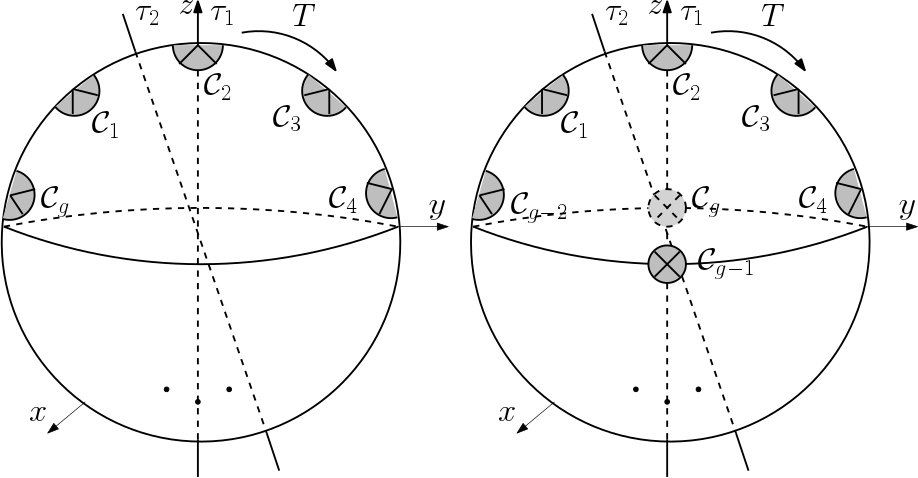}}
\caption{The involutions $\tau_1$ and $\tau_2$.}
\label{SO}
\end{center}
\end{figure}
\begin{figure}[h]
\begin{center}
\scalebox{0.3}{\includegraphics{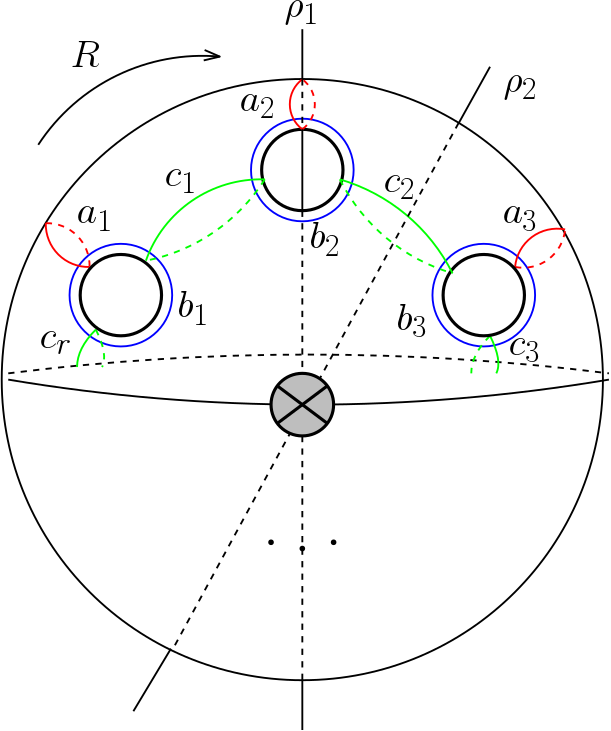}}
\caption{The involutions $\rho_1$ and $\rho_2$ for $g=2r+1$.}
\label{CODD}
\end{center}
\end{figure}
The model in Figure~\ref{SO} is the nonorientable surface obtained from $\mathbb{S}^{2}$ embedded in $\mathbb{R}^{3}$ 
and by deleting the interiors of $g$-disjoint disks and identifying the antipodal points on 
the boundary of each removed disks, say $\mathcal{C}_i$. Moreover, each crosscap 
$\mathcal{C}_i$ is in a circular position with the second crosscap $\mathcal{C}_2$ 
on the $+z$-axis and the rotation $T$ by $\frac{2\pi}{g}$ about $x$-axis maps 
the crosscap $C_i$ to $C_{i+1}$. The model in Figure~\ref{CODD} is obtained
 from a genus $r$ orientable surface by deleting the interior of a disk and identifying the antipodal points on the boundary. Moreover, the genus $r$ 
surface minus a disk is embedded in $\mathbb{R}^{3}$ in such a way that each genus 
is in a circular position with the second genus on the $+z$-axis and the rotation $R$ by
$\frac{2\pi}{r}$ about $x$-axis maps the curve $b_i$ to $b_{i+1}$ for
$i=1,\ldots,r-1$ and $b_r$ to $b_1$.\\
\noindent
We use the explicit homeomorphism constructed in \cite[Section 3]{st1} to identify the models in Figure~\ref{G} and Figure~\ref{CODD}. In Figure~\ref{CODD},
one crosscap is on the $+x$-axis. Note that the surface $N_g$ is invariant under the two involutions $\rho_{1}$ and $\rho_{2}$ 
where $\rho_{1}$ is the reflection in the $xz$-plane and $\rho_{2}$ is the reflection in the plane $z=tan(\frac{\pi}{r})y$ as in Figure~\ref{CODD}. The rotations $\rho_{1}$ and $\rho_{2}$ 
satisfy $D(\rho_{1})=D(\rho_{2})=1$ if $g=4k+1$. In this case, the twist subgroup $\mathcal{T}_g$ contains $\rho_{1}$ and $\rho_{2}$. Observe that the rotation $R=\rho_2\rho_1$.

For $g=4k+3$, we work with the model on the right hand side of Figure~\ref{SO}. This surface is a genus-$g$ 
nonorientable surface obtained from $\mathbb{S}^{2}$ embedded in $\mathbb{R}^{3}$ 
and by deleting the interiors of $g$-disjoint disks and identifying the antipodal points on 
the boundary of each removed disks, say $\mathcal{C}_i$. Moreover, each crosscap 
$\mathcal{C}_i$ for $i=1,\ldots,g-2$ is in a circular position with the second crosscap $\mathcal{C}_2$ 
on the $+z$-axis, the rotation $T$ by $\frac{2\pi}{g-2}$ about $x$-axis maps 
the crosscap $\mathcal{C}_i$ to $\mathcal{C}_{i+1}$ for $i=1,\ldots,g-3$. The crosscap  $\mathcal{C}_{g-1}$ is on the $+x$-axis and  $\mathcal{C}_g$ is obtained by rotating  $\mathcal{C}_{g-1}$ by $\pi$ about $+z$-axis. Note that the surface $N_g$ is invariant under the two reflections $\tau_1$ and 
$\tau_2$ where $\tau_1$ is the reflection in the $z$-axis and $\tau_2$ is the 
reflection in the plane  $z=tan(\frac{\pi}{r})y$ as in Figure~\ref{SO}. The 
reflections $\tau_1$ and $\tau_2$ satisfy $D(\tau_1)=D(\tau_2)=1$ if $r$ is 
even, which implies that  $\tau_1$ and $\tau_2$ are contained in the twist 
subgroup $\mathcal{T}_g$.\\
\noindent
Recall that in Theorem~\ref{t1} we give a generating set for $\mathcal{T}_g$ when $g$ is even. We have the following generators when $g$ is odd.
\begin{theorem}\label{t2}
Let $r\geq3$ and $g=2r+1$. Then the twist subgroup $\mathcal{T}_g$ is generated by 
the elements $R, A_1A_{2}^{-1}, B_1B_{2}^{-1}, C_1C_{2}^{-1}$ and $E$. 
\end{theorem}
\begin{proof}
Let $G$ be the subgroup of $\mathcal{T}_g$  generated by the set 
\[
\lbrace R, A_1A_{2}^{-1}, B_1B_{2}^{-1}, C_1C_{2}^{-1}, E\rbrace
\]
if $g=2r+1$.
Let $\mathcal{S}$ denote the set of isotopy classes of two-sided non-separating 
simple closed curves on $N_g$. Define a subset $\mathcal{G}$ of $\mathcal{S}\times \mathcal{S}$ 
as 
\[
\mathcal{G} =\lbrace(a,b): AB^{-1}\in G \rbrace.
\]
The set $\mathcal{G}$ defines an equivalence relation on $\mathcal{S}$ which satisfies 
$G$-invariance property, that is, 
\begin{center}
if $(a,b)\in \mathcal{G}$ and $H\in G$ then $(H(a),H(b))\in \mathcal{G}$.
\end{center}
Then it follows from the proof of Theorem~\ref{mt1} that the Dehn twists 
$A_i$ and $B_i$ for $i=1,\ldots,r$ are contained in $G$. Also, $G$ contains $C_j$ 
for $j=1,\ldots,r-1$. Since all generators given in Theorem~\ref{thm1} are contained 
in the group $G$. We conclude that $G=\mathcal{T}_g$.
\end{proof}

\begin{figure}[h]
\begin{center}
\scalebox{0.35}{\includegraphics{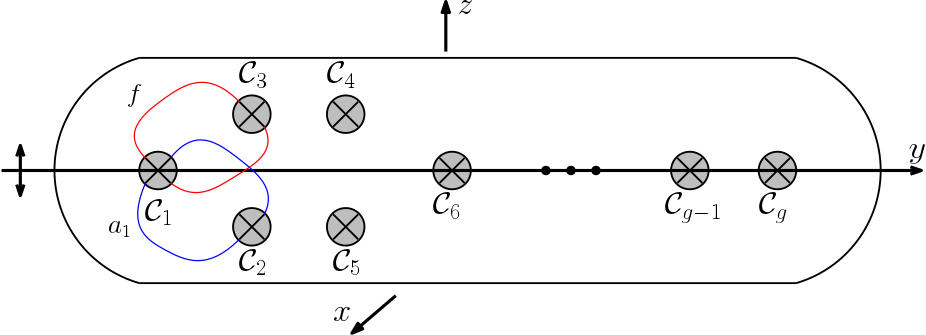}}
\caption{The involution $\beta$ for $g=2r+1$.}
\label{beta}
\end{center}
\end{figure}
Let $g=2r+1$ and consider the surface $N_g$, where $g$-crosscaps are distributed on $\mathbb{S}^2$ as in Figure~\ref{beta}. First, we introduce a reflection $\beta$ on $N_g$ in the $xy$-plane such that
\begin{itemize}
\item $\beta(a_1)=f$,
\item $\beta(x_2)=x_3$, $\beta(x_4)=x_5$ and
\item $\beta(x_1)=x_1$, $\beta(x_i)=x_i$ for $i=6,7,\ldots,g$.
\end{itemize}
The involution $\beta$ reverses the orientation of a neighbourhood of a two-sided simple closed curve. It satisfies $D(\beta)=1$ and hence $\beta$ is an element of $\mathcal{T}_g$.\\
\noindent
For the remaining generators of the following theorem we refer to Figures~\ref{CODD} and ~\ref{beta}. 
\begin{theorem}
For $g=4k+1$ and $k\geq3$, the twist subgroup $\mathcal{T}_g$ is generated by the 
four involutions  $\rho_1,\rho_2,\rho_1A_2C_{\frac{r}{2}}B_{\frac{r+4}{2}} C_{\frac{r+6}{2}}$ and $\beta$, where $r=2k$.   
\end{theorem}
\begin{proof}
Consider the surface $N_{g}$ as in Figure~\ref{CODD}. The involution $\rho_1$ satisfies
\[
 \rho_1(a_2)=a_2, \rho_1(b_{\frac{r+4}{2}})=b_{\frac{r+4}{2}} \textrm{ and } \rho_1(c_{\frac{r}{2}})=c_{\frac{r+6}{2}}.
\]
Since $\rho_1$ reverses the orientation of a neighbourhood of a two-sided simple closed curve, we have
\begin{itemize}
	\item $\rho_1A_2\rho_1=A_{2}^{-1}$ 
	\item $\rho_1B_{\frac{r+4}{2}}\rho_1=B_{\frac{r+4}{2}}^{-1}$ and
	\item $\rho_1C_{\frac{r}{2}}\rho_1=C_{\frac{r+6}{2}}^{-1}$.
	\end{itemize}
 It can be 
shown that $\rho_1A_2C_{\frac{r}{2}}B_{\frac{r+4}{2}} C_{\frac{r+6}{2}}$ 
is an involution. Let $H$ be the subgroup of $\mod(N_{g})$ generated by the set
	\[
	\lbrace \rho_1,\rho_2,\rho_1A_2C_{\frac{r}{2}}B_{\frac{r+4}{2}} C_{\frac{r+6}{2}},\beta\rbrace .
	\]
Observe that $R=\rho_1\rho_2\in K$. By the proof of Theorem~\ref{t3.5}, the elements $ A_1A_{2}^{-1},B_1B_{2}^{-1}$ and $,C_1C_{2}^{-1}$ belong to $H$. Since $A_1\beta(a_1)=e$, the element $E$ is in $H$. We conclude that $\mathcal{T}_g=H$ by Theorem~\ref{t2}.
\end{proof}
Although we give a generating set of $4$ involutions, for completeness of the applications of our method, first we give the following theorem.
\begin{theorem}
For $g=4k+1$ and $k\geq1$, the twist subgroup $\mathcal{T}_g$ is generated by the 
five involutions $\tau_1$, $\tau_2$, $\tau_1\tau_2\tau_1A_2$, $\tau_2A_1$ and $\beta$.
\end{theorem}
\begin{proof}
Let $K$ be the subgroup of $\mathcal{T}_g$ generated by the set
\[
\lbrace \tau_1,\tau_2, \tau_1\tau_2\tau_1A_2,\tau_2A_1,\beta   \rbrace.
\] 
Note that the rotation $T=\tau_1\tau_2$ is contained in $K$.
It follows from 
\begin{itemize}
	\item $A_1=\tau_2\tau_2A_1$ and
	\item $A_2=(\tau_1\tau_2)\tau_1(\tau_1\tau_2\tau_1)A_2$
\end{itemize}
that the elements $A_1$ and $A_2$ are in $K$. By conjugating $A_1$ 
with powers of $T$, $\mathcal{T}_g$ contains the elements $B_i$ and 
$C_i$. Moreover, it follows from $\beta(a_1)=f$ that the element $F$ is in $K$. Since $A_1(f)=e$, we get $E\in K$. 
This finishes the proof by Theorem~\ref{thm1}.
\end{proof}
In the next theorem, we present four involutions to generate particularly $\mathcal{T}_5$ and $\mathcal{T}_9$. This completes the case $g=4k+1$ and $k\geq1$. First, recall that $A_1,A_2,B_1,B_2,C_1$ and $E$ generate $\mathcal{T}_5$ and $A_1,A_2,B_1,B_2,B_3,B_4$,
$C_1,C_2,C_3$ and $E$ generate $\mathcal{T}_9$. We use the following three involutions $\gamma,S\gamma$ and $S^{2k-2}(S\gamma)S^{2-2k}A_2$ of the generating set given in \cite[Theorem 5]{sz2}. The involution $\gamma$ is defined as the reflection in the $xz$-plane where the crosscaps are distributed along the equator on $\mathbb{S}^{2}$. The map $S$ is defined as the composition
$B_{2k}C_{2k-1}B_{2k-1}\cdots C_1B_1A_1$.  Note that $D(\gamma)=D(S\gamma)=D(S^{2k-2}(S\gamma)S^{2-2k}A_2)=1$. 
\begin{theorem}
The twist subgroups $\mathcal{T}_5$ and $\mathcal{T}_9$
can be generated by the involutions $\gamma,S\gamma, S^{2k-2}(S\gamma)S^{2-2k}A_2$ and $\beta$ for $k=1,2$.  
\end{theorem}
\begin{proof}
The generator $A_1$ can be obtained by $S$ and $A_2$~\cite[Theorem 5]{mk2}. By conjugating with  powers of $S$, it is easy to see that the elements $B_i$ and $C_i$ belong to $\mathcal{T}_g$. Also, the generator $E$ is contained in $\mathcal{T}_g$ since $A_1\beta(a_1)=e$.
\end{proof}
\begin{figure}[h]
\begin{center}
\scalebox{0.4}{\includegraphics{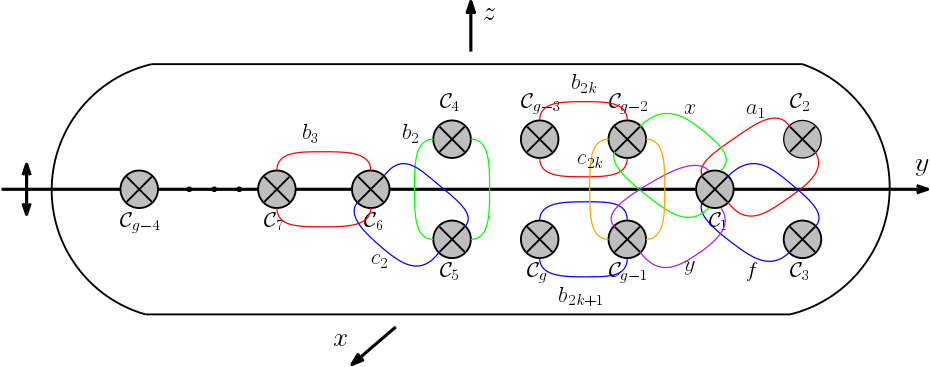}}
\caption{The involution $\mu$ for $g=4k+3$.}
\label{MU}
\end{center}
\end{figure}
Now, let $g=4k+3$ and consider $N_g$, where $g$-crosscaps are distributed over $\mathbb{S}^2$ as in Figure~\ref{MU}. The surface $N_g$ is symmetrical in the $xy$-plane. Let $\mu$ be the reflection in the $xy$-plane. Note that the linear map associated to the involution $\mu$ satisfies $D(\mu)=1$ if $k\geq2$. Therefore, the involution $\mu$ is in $\mathcal{T}_g$ for $k\geq2$.
\begin{theorem}
For $g=4k+3$ and $k\geq2$, the twist subgroup $\mathcal{T}_g$ is generated by the 
five involutions $\tau_1$, $\tau_2$, $\tau_1\tau_2\tau_1A_2$, $\tau_2A_1$ and $\mu$.
\end{theorem}
\begin{proof}
Let $K$ be the subgroup of $\mathcal{T}_g$ generated by the set
\[
\lbrace \tau_1,\tau_2, \tau_1\tau_2\tau_1A_2,\tau_2A_1,\mu \rbrace.
\] 
Note that the rotation $T=\tau_1\tau_2$ is contained in $K$.
It follows from 
\begin{itemize}
	\item $A_1=\tau_2(\tau_2A_1)$ and
	\item $A_2=(\tau_1\tau_2\tau_1)(\tau_1\tau_2\tau_1A_2)$
\end{itemize}
that the elements $A_1$ and $A_2$ are in $K$. By conjugating $A_1$ 
with powers of $T$, $\mathcal{T}_g$ contains the elements $B_i$ for $i=1,\ldots, 2k$ and 
$C_j$ for $j=1,\ldots, 2k-1$. \\
\noindent
Let $T(b_{2k})=x$ and $\mu(x)=y$. Then the elements $X$
and $Y$ are contained in $K$ by the fact that $B_{2k}$ is in $K$.\\
\noindent
It follows from 
\begin{itemize}
\item $T^{-1}(y)=c_{2k}$,
\item $\mu(b_{2k})=b_{2k+1}$
\end{itemize}
that $C_{2k}$ and $B_{2k+1}$ are contained in $K$. This completes the proof by Theorem~\ref{thm1}.
\end{proof}
\begin{figure}[h]
\begin{center}
\scalebox{0.4}{\includegraphics{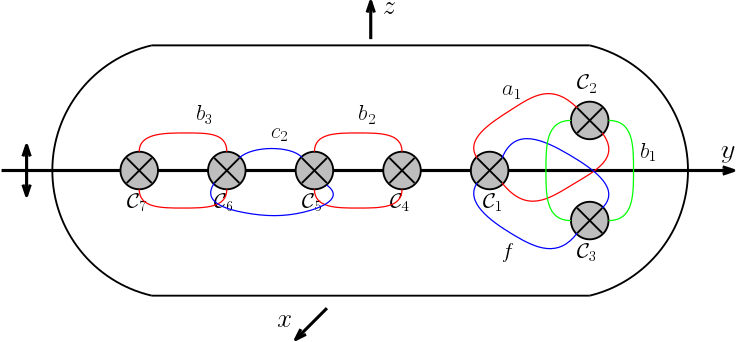}}
\caption{The involution $\sigma_1$ for $g=7$.}
\label{SP}
\end{center}
\end{figure}
\begin{figure}[h]
\begin{center}
\scalebox{0.4}{\includegraphics{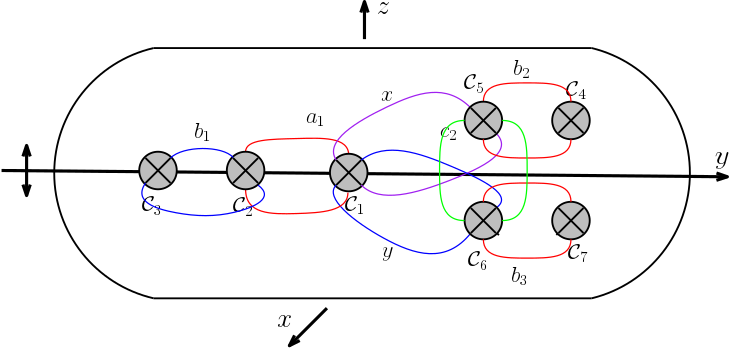}}
\caption{The involution $\sigma_2$ for $g=7$.}
\label{SDP}
\end{center}
\end{figure}
For the surface $N_7$, we introduce two involutions, $\sigma_1$ and $\sigma_2$, shown in Figure~\ref{SP} and ~\ref{SDP}. In these figures, the surface is symmetric with respect to the $xy$-plane. Both $\sigma_1$ and $\sigma_2$
are reflections in the $xy$-plane and $D(\sigma_1)=D(\sigma_2)=1$. Hence, both $\sigma_1$ and $\sigma_2$ belong to $\mathcal{T}_7$. For the remaining generators in the following theorem we refer to the model on the right hand side of Figure~\ref{SO}.
\begin{theorem}
The twist subgroup $\mathcal{T}_7$ is generated by the 
six involutions $\tau_1$, $\tau_2$, $\tau_1\tau_2\tau_1A_2$, $\tau_2A_1$, $\sigma_1$ and $\sigma_2$.
\end{theorem}
\begin{proof}
Let $K$ be the subgroup of $\mathcal{T}_g$ generated by the set
\[
\lbrace \tau_1,\tau_2, \tau_1\tau_2\tau_1A_2,\tau_2A_1,\sigma_1,\sigma_2 \rbrace.
\] 
Note that the rotation $T=\tau_1\tau_2$ is contained in $K$.
It follows from 
\begin{itemize}
	\item $A_1=\tau_2(\tau_2A_1)$ and
	\item $A_2=(\tau_1\tau_2\tau_1)(\tau_1\tau_2\tau_1A_2)$
\end{itemize}
that the elements $A_1$ and $A_2$ are in $K$. By conjugating $A_1$ 
with powers of $T$, $\mathcal{T}_g$ contains the elements $B_1,C_1$ and $B_2$. \\
\noindent
Let $T(b_{2})=x$ and $\sigma_2(x)=y$. Then the elements $X$
and $Y$ are contained in $K$. It follows from 
\begin{itemize}
\item $T^{-1}(y)=c_{2}$,
\item $\sigma_2(b_{2})=b_{3}$
\end{itemize}
that $C_{2}$ and $B_{3}$ are contained in $K$. Moreover, since $A_1\sigma_1(a_1)=e$, $E\in K$, which completes the proof by Theorem~\ref{thm1}.
\end{proof}


\end{document}